\documentclass[a4paper,11pt,reqno]{amsart}
\usepackage{amsmath, amssymb, amsthm, mathrsfs}
\usepackage{hyperref, mathtools, bm}
\usepackage[in]{fullpage}
\usepackage{cite}
\usepackage{hyperref,xcolor}
\usepackage{fancyhdr}

\numberwithin{equation}{section}

\definecolor{dark-green}{rgb}{0.1,0.1,0.3}

\hypersetup{
colorlinks=true,
    pdfborder={0 0 0},
citecolor=dark-green,
linkcolor=blue
}

\newcommand{\floor}[1]{\left\lfloor #1 \right\rfloor}

\newcommand{\ceil}[1]{\left\lceil #1 \right\rceil}

\newcommand{\sg}{\sigma}

\newcommand{\al}{\alpha}

\newcommand{\mb}{\mathbb}

\newcommand{\bt}{\beta}

\newcommand{\lm}{\limits}

\newcommand{\st}{\substack}

\newcommand{\sr}{\sqrt}

\newcommand{\fr}{\frac}

\newcommand{\tps}{\texorpdfstring}

\newcommand{\F}{\mathscr{F}}

\newcommand{\N}[1]{\mb{N}^{#1}}

\newcommand{\Z}[1]{\mb{Z}^{+}(#1)}

\newtheorem{thm}{Theorem}[section]

\newtheorem{cor}[thm]{Corollary}

\newtheorem{lem}[thm]{Lemma}

\newtheorem{prop}[thm]{Proposition}

\theoremstyle{definition}

\newtheorem{defn}[thm]{Definition}

\newtheorem*{ntn}{Notation}

\newtheorem{rem}[thm]{Remark}

\newtheorem{exm}[thm]{Example}

\begin{document}
\title{On the number of factorizations of an integer}

\author{R. Balasubramanian}
\author{Priyamvad Srivastav}
\address{Institute of Mathematical Sciences\\ Taramani\\ Chennai\\ India-600113 and Homi Bhabha National Institute\\ Training School Complex\\ Anushakti Nagar\\ Mumbai\\ India-400094.}
\email[R. Balasubramanian]{balu@imsc.res.in}
\email[Priyamvad Srivastav]{priyamvads@imsc.res.in}

\date{}

\begin{abstract}
Let $f(n)$ denote the number of unordered factorizations of a positive integer $n$ into factors larger than $1$. We show that the number of distinct values of $f(n)$, less than or equal to $x$, is at most 
$\exp \left( C \sqrt{\frac{\log x}{\log \log x}} \left(1 + o(1) \right) \right)$, 
where $C=2\pi\sqrt{2/3}$ and $x$ is sufficiently large. This improves upon a previous result of the first author and F. Luca.
\end{abstract}




\maketitle

\section{Introduction}
\thispagestyle{empty}
Let $f(n)$ denote the number of unordered factorizations of $n$ into factors larger than $1$. More precisely, $f(n)$ is the number of tuples $(n_1,\dots, n_r)$, such that $1<n_1 \leq n_2 \leq \dots \leq n_r$ and $n=n_1 n_2 \dots n_r$. 
For example, $f(18)=4$, since $18$ has the factorizations
$$18, \quad 2 \cdot 9, \quad 3 \cdot 6, \quad 2 \cdot 3 \cdot 3.$$ 
The function $f(n)$ is a multiplicative analogue of the the partition function. \vskip 0.04in
There are various results on the properties of this function. The problem of determining the exact nature of $f(n)$ was considered by Oppenheim \cite{1}. He proved that
\begin{equation} 
\sum\lm_{n \leq x} f(n) \sim \fr{x \exp(2 \sr{\log x})}{2 \sr{\pi} (\log x)^{3/4}}.
\end{equation}
\medskip 
Further investigation was carried out by E.R. Canfield, P. Erd\H{o}s and C. Pomerance \cite{2}, who showed that the maximal order of $f(n)$ is
\begin{equation} 
n \exp\left( (-1 + o(1)) \fr{l_1(n)l_3(n)}{l_2(n)} \right),
\end{equation}
where $l_k(n)$ is the $k$-fold iteration of the natural logarithm. \\ \vskip 0.05in
\begin{defn}
\label{F}
For any $x \geq 1$, let $\mathscr{F}(x)$ be the the set of values of $f(n)$, not exceeding $x$, i.e. 
\begin{equation} 
\F(x) = \left\{f(n): f(n) \leq x \right\}.
\end{equation}
\end{defn}

\medskip

In \cite{2}, the authors claimed that they could prove 
$ \# \F(x) = x^{o(1)}$, as $x \to \infty$. In this connection, F. Luca, A. Mukhopadhyay and K. Srinivas \cite{3}, proved that
\begin{equation} 
\# \F(x) = x^{O\left( \log \log \log x / \log \log x \right)}.
\end{equation}
\medskip
This bound was improved in \cite{4} by the first author and F. Luca. They proved
\begin{equation} 
\# \F(x) \leq \exp\left(9(\log x)^{2/3}\right), \quad \text{for all} \ x \geq 1.
\end{equation}
\vskip 0.07in
In this paper, we further improve the above result. We prove
\begin{thm}
\label{main}
Let $C = 2 \pi \sr{2/3}$ and $x$ be sufficiently large. Then
$$ \# \F(x) \leq \exp\left(C\sr{\fr{\log x}{\log \log x}} \left( 1 + O\left( \fr{\log \log \log x}{\log \log x} \right) \right) \right).$$
\end{thm}
\medskip
We have strong reasons to believe that up to a constant, the above bound for $\log \# \F(x)$ is the best possible. We will discuss more on this in the final section.  
   
\bigskip

\section{Outline of the proof}
In \cite{1} and \cite{2}, the following observations were made:
\medskip

\begin{itemize}
\item[(1)] For any prime $q$,
\begin{equation} 
\label{obs1}
f(q^n)=p(n),
\end{equation} 
where $p(n)$ is the partition function.
\item[(2)] If $p_1, p_2, \dots, p_r$ are distinct primes, then
\begin{equation}
\label{obs2}
f(p_1 \dots p_r) = B_r,
\end{equation}
where $B_r$ is the $r^{\text{th}}$ \emph{Bell number}, which is also the number of partitions of a set having $r$ distinct elements. 
\end{itemize}
\medskip

In view of these observations, we define a generalization of the partition function to the elements of $\N{r}$. 
\begin{ntn}
For any $r \geq 1$, let
\begin{equation}
\label{Z+(r)}
\Z{r}: = (\mb{Z}_{\geq 0})^r \setminus \{\bm{0}\}, \quad \text{where} \ \bm{0}=(0, \dots,0).
\end{equation}
\end{ntn}

\medskip

\begin{defn}
\label{p}
Let $\bm{\al} = (\al_1,\dots,\al_r) \in \N{r}$. A partition of $\bm{\al}$ is an unordered decomposition
$$ \bm{\al} = \bm{\bt_1} + \dots + \bm{\bt_l},$$
where $\bm{\bt_i} \in \Z{r}$, for each $1 \leq i \leq l$ and the addition is component-wise. The number of partitions of $\bm{\al}$ is denoted by $p(\bm{\al})$.
\end{defn}
\begin{exm}
The partitions of $\bm{\al} = (1,2)$ are
$$ (1,2), \quad (1,0) + (0,2), \quad (0,1) + (1,1), \quad (0,1) + (0,1) + (1,0).$$ 
\end{exm}
\begin{rem}
When $r=1$, the above corresponds to the usual \emph{partition function} in $\mb{N}$. Moreover, any such partition $\pi$ 
of $\bm{\al} \in \N{r}$ can be represented as 
$$\pi = \prod\lm_{\bm{\bt} \in \Z{r}} \bm{\bt}^{\pi(\bm{\bt})},$$ 
as in the case $r=1$.
\end{rem}

\begin{rem}
The above function can also be thought of as a partition of the multi-set 
$$\{1,1, \dots , 1, 2, \dots , 2, \dots, r, \dots , r\},$$ 
with each $i$ having exactly $\al_i$ copies, for $1 \leq i \leq r$. When $\al_i=1$ for each $i$, this corresponds to a set-partition, the number of which is given by the $r^{\text{th}}$ \emph{Bell number} $B_r$.
\end{rem}
\medskip
The following lemma generalizes the observations in (\ref{obs1}) and (\ref{obs2}).
\begin{lem}
Let $n=p_1^{\al_1} \dots p_r^{\al_r}$ and $\bm{\al}=(\al_1, \dots, \al_r)$. Then
$$ f(n) = p(\bm{\al}).$$
\end{lem}

\begin{proof}
Let $n=n_1 n_2 \dots n_l$ be a nontrivial factorization of $n$, with $n_i>1$ for each $i$. For each $1 \leq i \leq l$, let 
$$ n_i = \prod\lm_{j=1}^r p_j^{\bt_{ij}} \quad \text{and} \quad \bm{\bt_i} = (\bt_{i1}, \dots, \bt_{ir}).$$ 
Then, clearly $\bm{\bt_i} \in \Z{r}$ and 
$\sum\lm_{i=1}^l \bm{\bt_i} = \bm{\al}$. Therefore, each unordered factorization gives rise to a partition of $\bm{\al}$. Clearly, the partition obtained in this way is unique. 
The converse follows analogously.
\end{proof}
\medskip
Therefore, $\# \F(x)$ is bounded above by the number of unordered tuples $\bm{\al}=(\al_1, \dots, \al_r)$, which satisfy $p(\bm{\al}) \leq x$. We record this as the following Corollary:

\begin{cor}
\label{fp}
$$ \# \F(x) \leq  \# \{1 \leq \al_1 \leq \dots \leq \al_r: \ p(\bm{\al}) \leq x \}.$$
\end{cor}
\medskip
The problem has now reduced to determining the distribution of $p(\bm{\al}) \leq x$. Therefore, we seek a lower bound for $p(\bm{\al})$. 

\begin{prop}
\label{lb}
Let $\bm{\al}=(\al_1, \dots, \al_r) \in \N{r}$. For any $z>0$, let 
\begin{equation}
\label{g(z)}
g(\bm{\al},z) = z \prod\lm_{i=1}^r \left( 1+\fr{\al_i}{z} \right)^{-1}.
\end{equation}
Then $g(\bm{\al},z)$ is a strictly increasing function whose value at $1$ is less than $1$. Let $z(\bm{\al})>1$ be the unique positive real solution to the equation $g(\bm{\al,z})=1$ and let $N=N(\bm{\al})$, be the greatest integer less than or equal to $z(\bm{\al})$, i.e., $N= \floor{z(\bm{\al})} \geq 1$. Then
\begin{itemize}
\item[(a)]
$$p(\bm{\al}) \geq  \fr{e^{N-2}}{2 \, N^{\fr{3}{2}}} 
\prod\lm_{i=1}^r \fr{1}{2\sr{2N}} \left( 1 + \fr{N}{\al_i} \right)^{\al_i + \fr{1}{2}}.$$


\item[(b)] Further, if $p(\bm{\al}) \leq x$, then for $x$ sufficiently large, we have
$$ r \leq R= \fr{2\log x}{\log \log x}\left( 1 + \fr{2\log \log \log x}{\log \log x} \right) \quad \text{and} \quad 
N \leq 3 \log x.$$
\end{itemize}
\end{prop}

\begin{ntn}
The quantity $N=N(\bm{\al})$ depends entirely on $\bm{\al}$. For sake of simplicity, we write this as $N$.
\end{ntn}
\medskip
We now prove Theorem \ref{main} using Proposition \ref{lb}. We assume throughout, that $x$ is sufficiently large. \vskip 0.07in
Let $\bm{\al} \in \N{r}$ be such that $p(\bm{\al}) \leq x$. Taking logarithm in the inequality in Proposition \ref{lb} (a), and transferring the negative terms to RHS, we obtain
\begin{equation*}
N + \sum\lm_{i=1}^r (\al_i + 0.5) \log \left( 1 + \fr{N}{\al_i}\right) \leq \log x + 0.5 (r+3) \log N + 1.04 \, r + 2.7 .
\end{equation*}
Using the bounds for $N$ and $r$ from Proposition \ref{lb} (b) in the RHS above, and simplifying, we get
\begin{equation}
\label{maineq}
\begin{split}
\sum\lm_{i=1}^r \al_i \log\left( 1+ \fr{N}{\al_i} \right) &\leq  2 \log x \left( 1 + O\left( \fr{\log \log \log x}{\log \log x} \right) \right).
\end{split}
\end{equation}
\medskip
Next, we split the set $\{\al_1,\dots, \al_r\}$ into two parts $I$ and $J$, where
$$I=\{\al_i: \al_i \leq A(N+1)\} \quad \text{and} \quad J=\{\al_i: \al_i>A(N+1)\},$$
and $A>0$ is a positive constant. We shall choose 
\begin{equation}
\label{A}
A = \fr{(\log \log x)^6}{(\log x)^{1/2}}.
\end{equation}
We separately estimate the number of choices for elements in $I$ and $J$. \vskip 0.02in
For elements of $I$,  we have $\al_i \leq A(N+1)$. Therefore, it follows that 
$$ \log \left( 1 + \fr{N}{\al_i} \right) \geq \log\left(1+\fr{N}{A(N+1)}\right) \geq \log\left(1+\fr{1}{2A}\right) 
\geq \fr{\log \log x}{2}  \left( 1 + O\left(\fr{\log \log \log x}{\log \log x}\right) \right).$$
for all $\al_i \in I$. With this applied to (\ref{maineq}), we obtain (ignoring the elements of $J$)
\begin{equation}
\label{tent} 
\sum\lm_{I} \al_i \leq \fr{4 \log x}{\log \log x} \left( 1 + O\left( \fr{\log \log \log x}{\log \log x} \right) \right).
\end{equation}
\medskip
The following lemma gives us the required upper bound for the number of such $\al_i$.
\begin{lem}
\label{maxp}
The number of unordered tuples $(n_1,\dots,n_l)$ of positive integers, for which
$$  \sum\lm_{i=1}^{l} n_i \leq y,$$
is at most $y\exp\left(\pi\sr{2y/3}\right)$, for all $y \geq 1$.
\end{lem}
\begin{rem}
The bound for the number of solutions above is actually $O(\sr{y} \exp(\pi\sr{2y/3}))$. As this is not quite useful for us, we keep the bound as above to make the proof easier.
\end{rem}
\begin{proof}[Proof of Lemma \ref{maxp}]
Suppose that $\sum\lm_{i=1}^{l} n_i = n \leq y$. From the proof of Theorem 15.3 in \cite[Pg 468]{6}, we have the upper bound
$$ p(n) \leq \exp\left( \pi \sr{2n/3} \right), \quad \text{for all} \ n \geq 1.$$
Therefore, the total number of choices for $n_1, \dots, n_l$ is at most
\begin{equation*}
\sum\lm_{n \leq y} \exp\left( \pi \sr{2n/3} \right) \leq y \exp\left( \pi \sr{2y/3} \right).
\end{equation*}  
\end{proof}

\medskip

Applying Lemma \ref{maxp} to (\ref{tent}), the total number of choices for $\al_i$'s in $I$, is at most
\begin{equation}
\label{maxr1}
\exp\left( 2 \pi \sr{\fr{2\log x}{3\log \log x}} \left( 1 + O\left( \fr{\log \log \log x}{\log \log x} \right) \right) \ \right).
\end{equation}
\vskip 0.08in
Next, we estimate the total number of choices for elements of $J$. Observe that for any $1 \leq i \leq r$, we have 
$p(\al_i) \leq p(\bm{\al}) \leq x$. Moreover, from Corollary 3.1 of \cite{5}, we also have the lower bound
$$ p(n) \geq \fr{\exp(2\sr{n})}{14}, \quad \text{for all} \ n \geq 1.$$
Therefore, in particular, for each $\al_i \in J$, we have
\begin{equation}
\label{log2x}
\al_i \leq \fr{1}{4} (\log 14x)^2 \leq \log^2 x,
\end{equation}
\medskip
In the next lemma, we estimate the cardinality of $J$.
\begin{lem}
\label{r2}
With $J$ as before, we have
$$ \# J  \leq \fr{4 \sr{\log x}}{(\log \log x)^5}.$$
\end{lem}
\begin{proof}
Note that $g(\bm{\al},z)$ is strictly increasing by Proposition \ref{lb}, with $z(\bm{\al})$ being the unique positive real solution to 
$g(\bm{\al},z)=1$. As $N \leq z(\bm{\al}) \leq N+1$, we have $g(\bm{\al},N+1) \geq 1$. Therefore
$$ N+1 \geq \prod\lm_{i=1}^r \left( 1 + \fr{\al_i}{N+1} \right) \geq \prod\lm_{\al_i \in J} \left( 1 + \fr{\al_i}{N+1} \right) 
\geq (1+A)^{\# J},$$ 
since $\al_i>A(N+1)$, for all $\al_i \in J$. \vskip 0.04in
Since $A<1$, we have $\log(1+A) \geq A/2$ and from Proposition \ref{lb}, we have $\log(N+1) \leq \log(1+3\log x) \leq 2 \log \log x$. Hence
$$ \#J \leq \fr{\log(N+1)}{\log(A+1)} \leq \fr{4 \sr{\log x}}{(\log \log x)^5}.$$
 This proves the lemma.
\end{proof}
\medskip
From (\ref{log2x}) and Lemma \ref{r2}, the number of choices for elements of $J$ is at most
\begin{equation}
\label{maxr2}
\left( \log^2 x \right)^{\#J} \leq \exp\left( \fr{8 \sr{\log x}}{(\log \log x)^4} \right),
\end{equation}
\medskip
Therefore, from (\ref{maxr1}) and (\ref{maxr2}), the total number of choices for $\bm{\al}$ is at most
\begin{equation*}
\quad \exp\left( 2\pi\sr{2/3} \sr{ \fr{ \log x}{ \log \log x}} \left( 1 + O\left( \fr{\log \log \log x}{\log \log x} \right) \right)  \right).
\end{equation*}
\medskip	
This completes the proof of Theorem \ref{main}. \vskip 0.07in
It now remains to give a proof of Proposition \ref{lb}.
\bigskip
\section{Preliminary lemmas}
In this section, we prove some Preliminary results. 
\subsection{Bounds on factorials and binomials}
We begin with the following lemma.
\begin{lem}
\label{converge}
Let 
\begin{equation*}
h_1(x)=\left(1+\fr{1}{x}\right)^{x+\fr{1}{2}}, \quad h_2(x)=\fr{x+1}{x+2} \left(1+\fr{1}{x}\right)^{x+\fr{3}{2}}.
\end{equation*}
Then, as $x \to \infty$, the functions $h_1$ and $h_2$ converge to $e$ decreasingly.  
\end{lem}
\medskip
Next, we obtain bounds for factorials and binomial coefficients.
\begin{lem}
\label{binofact}
Let $n$ and $k$ be positive integers. Then
\begin{itemize}
\item[(a)] 
$$(k+1)! \leq \fr{2 \, k^{k+\fr{3}{2}}}{e^{k-1}},$$
\item[(b)]
$$ \binom{k+n}{k} \geq \fr{1}{2\sr{2}} \fr{(k+n)^{k+n+\fr{1}{2}}}{k^{k+\fr{1}{2}} n^{n+\fr{1}{2}}}.$$
\end{itemize}
\end{lem}
\begin{proof}
Proof is by induction on $k$. We first prove (a). \vskip 0.02in
When $k=1$, (a) is trivially true. So, assume that (a) holds for some $k \geq 1$. Then, by induction
\begin{equation}
\label{temp_a} 
(k+2)! = (k+2) (k+1)! \leq \fr{2(k+2) k^{k+\fr{3}{2}}} {e^{k-1}}.
\end{equation}
We need to show that the RHS of (\ref{temp_a}) is at most
$$\fr{2 (k+1)^{k+\fr{5}{2}}}{e^k},$$
which is equivalent to
$$ \fr{k+1}{k+2}\left( 1+\fr{1}{k} \right)^{k + \fr{3}{2}} \geq e,$$
and this is true by Lemma \ref{converge} for the function $h_2$. \vskip 0.07in
Next, we prove (b). When $k=1$, this reduces to
$$ \left( 1+\fr{1}{n} \right)^{n+\fr{1}{2}} \leq 2\sr{2}.$$
This is true from Lemma \ref{converge}, since the function $h_1$ is decreasing and therefore its maximum on the positive integers is attained at $n=1$. \vskip 0.04in
Now, suppose that the (b) holds true for $(k,n)$. Then, by induction
\begin{equation}
\label{temp_b}
\binom{k+n+1}{k+1} =\fr{k+n+1}{k+1} \binom{k+n}{k} \geq 
\fr{1}{2\sr{2}}\fr{(k+n+1)}{(k+1)}\fr{(k+n)^{k+n+\fr{1}{2}}}{k^{k+\fr{1}{2}} n^{n+\fr{1}{2}}}.
\end{equation}
We need to show that the RHS of (\ref{temp_b}) is at least 
$$\fr{1}{2\sr{2}} \fr{(k+n+1)^{k+n+\fr{3}{2}}}{(k+1)^{k+\fr{3}{2}} n^{n+\fr{1}{2}}}.$$
This is equivalent to
$$ \left( 1 + \fr{1}{k} \right)^{k+ \fr{1}{2}} \geq \left( 1 + \fr{1}{k+n} \right)^{k+n+\fr{1}{2}},$$
which is true since $h_1$ is decreasing from Lemma \ref{converge}.
This completes the proof.
\end{proof}
\medskip
\subsection{A generating function for $\tps{p(\bm{\al})}{}$}
We give a generating function for $p(\bm{\al})$, which we later use to obtain a lower bound for $p(\bm{\al})$. We use the following notation:
\begin{ntn}
Let $\bm{q}=(q_1,\dots, q_r)$, with $|q_i| < 1$ for each $1 \leq i \leq r$. For $\bm{\bt} \in \Z{r}$, we use the notation
$$ \bm{q}^{\bm{\bt}} := q_1^{\bt_1} \dots q_r^{\bt_r}.$$
\end{ntn}
\medskip
We have
\begin{lem}
\label{generate}
Let $$ P(\bm{q}) = \prod\lm_{\bm{\bt} \in \Z{r}} \left( 1 - \bm{q}^{\bm{\bt}} \right)^{-1}.$$
Then $P(\bm{q})$ is a generating function for $p(\bm{\al})$ i.e., for any $\bm{\al} \in \N{r}$, the coefficient of 
$\bm{q}^{\bm{\al}}$ in $P(\bm{q})$ is $p(\bm{\al})$. 
\end{lem}
\begin{rem}
When $r=1$, the above corresponds to the generating function of the partition function $p(n)$.
\end{rem}
\begin{proof}[Proof of Lemma \ref{generate}]
Since the given product converges locally uniformly, we can write it as
\begin{equation}
\begin{split}
P(\bm{q}) &= \prod\lm_{\bm{\bt} \in \Z{r}} \left( \sum\lm_{l=0}^{\infty}  \bm{q}^{l \bm{\bt}} \right)
=\sum\lm_{h: \Z{r} \to \mb{Z}_{\geq 0} }\bm{q}^{h(\bm{\bt})\cdot\bm{\bt} }
\end{split}
\end{equation}
Therefore, the coefficient of $\bm{q}^{\bm{\al}}$ above equals the number of all functions $h :\Z{r} \to \mb{Z}_{\geq 0}$, for which 
$$\sum\lm_{\bm{\bt} \in \Z{r}} h(\bm{\bt}) \cdot\bm{\bt} = \bm{\al}.$$
We show that the above quantity equals $p(\bm{\al})$. Suppose that $\pi$ is a partition of $\bm{\al}$. Then one can write $\pi$ as 
$$ \pi = \prod\lm_{\bm{\bt} \in \Z{r}} \bm{\bt}^{h(\bm{\bt})}.$$
Clearly, the above gives rise to a unique such function $h$. Conversely, any such function $h$ gives a unique product decomposition as above. This completes the proof.
\end{proof}
\medskip
We prove the following lemma about the exponential of a power series:
\begin{lem}
\label{exp+}
Suppose that 
$$F(\bm{q})=a(\bm{0})+\sum\lm_{\bm{n} \in \Z{r}}^{\infty} a(\bm{n}) \bm{q}^{\bm{n}},$$ 
is convergent in $\{\bm{q}:|q_i|<1 \}$, with real coefficients satisfying $a(\bm{n}) \geq 0$, for $\bm{n} \in \Z{r} \cup \{\bm{0}\}$. Then the power series of $G(\bm{q})=\exp(F(\bm{q}))$ around $\bm{0}$ also has non-negative coefficients. 
\end{lem}
\begin{proof}
Note that
\begin{equation*}
G(\bm{q}) = \sum\lm_{k=0}^{\infty} \fr{F(\bm{q})^k}{k!}.
\end{equation*}
Now, since $a(\bm{n}) \geq 0$, for each $\bm{n} \in \Z{r}$, it follows that the coefficients of $F(\bm{q})^k$ are non-negative for each $k \geq 0$. Therefore, $G(\bm{q})$ has non-negative coefficients.
\end{proof}
\medskip
Next, we obtain a lower bound for $p(\bm{\al})$.
\begin{lem}
\label{hypergeometric}
Let $\bm{\al} \in \N{r}$. Then
\begin{equation}
\label{hg} p(\bm{\al}) \geq \fr{1}{e}\sum\lm_{k=0}^{\infty} \fr{1}{(k+1)!} \prod\lm_{i=1}^r \binom{k+\al_i}{k}.
\end{equation}
\end{lem}
\begin{rem}
The RHS of (\ref{hg}) can be written in terms of a generalized hypergeometric series as
$$\fr{1}{e} \ {}_r F_r\left( \begin{matrix} \al_1 + 1 &\dots&\dots& \al_{r-1}+1 & \al_r+1\\1 & \dots & \dots & 1 & 2 \end{matrix} \ ; \, 1 \right).$$
When $\bm{\al} = (1,1, \dots, 1)$, equality holds in (\ref{hg}) and the RHS of (\ref{hg}) becomes the Dobi\'nski's formula for the $r^{\text{th}}$ Bell number $B_r$.
\end{rem}
\begin{proof}[Proof of Lemma \ref{hypergeometric}]
Taking logarithms in the expression for $P(\bm{q})$ in Lemma \ref{generate}, we get
\begin{equation}
\begin{split}
\label{log}
\log P(\bm{q}) &= \sum\lm_{\bm{\bt} \in \Z{r}} - \log(1-\bm{q}^{\bm{\bt}})=\sum\lm_{\bm{\bt} \in \Z{r}} \sum\lm_{m=1}^{\infty} \fr{\bm{q}^{m \bm{\bt}}}{m}= \sum\lm_{\bm{\bt} \in \Z{r}} \bm{q}^{\bm{\bt}} \sum\lm_{m \mid \bt_i \forall i} \fr{1}{m}\\ 
&= \sum\lm_{\bm{\bt} \in \Z{r}} \fr{\sg(\bt_1,\dots,\bt_r)}{(\bt_1,\dots,\bt_r)}\bm{q}^{\bm{\bt}}\\
&= \sum\lm_{\bm{\bt} \in \Z{r}} \bm{q}^{\bm{\bt}} \  +  \  H(\bm{q}),
\end{split}
\end{equation}
where $\sg(\bt_1,\dots,\bt_r)$ denotes $\sg(\gcd(\bt_1,\dots,\bt_r))$, and
\begin{equation}
\label{H}
H(\bm{q}) = \sum\lm_{\bm{\bt} \in \Z{r}} \left( \fr{\sg(\bt_1,\dots,\bt_r)}{(\bt_1,\dots,\bt_r)}  - 1 \right) \bm{q}^{\bm{\bt}}.
\end{equation}
\medskip
Taking exponential in (\ref{log}), we get
\begin{equation}
P(\bm{q}) = \exp\left( \sum\lm_{\bm{\bt} \in \Z{r}} \bm{q}^{\bm{\bt}} \right) \cdot \exp(H(\bm{q})).
\end{equation}
\medskip
Now, we have
\begin{equation}
\label{1st}
\sum\lm_{\bm{\bt} \in \Z{r}} \bm{q}^{\bm{\bt}}=\sum\lm_{\st{\bt_i, \dots, \bt_r \geq 0}} q_1^{\bt_1} \dots q_r^{\bt_r} - 1=\fr{1}{(1-q_1) \dots (1-q_r)} - 1.
\end{equation}
\medskip
Note that $H(\bm{q})$ has non-negative coefficients with constant term $0$. Therefore, by Lemma \ref{exp+}, $\exp(H(\bm{q}))$ also has non-negative coefficients with constant term $1$. Therefore, the coefficient of $\bm{q}^{\bm{\al}}$ in $P(\bm{q})$ is at least $1/e$ times the  coefficient of $\bm{q}^{\bm{\al}}$ in $\exp\left( \prod\lm_{i=1}^r (1-q_i)^{-1} \right)$. \vskip 0.05in
Since
\begin{equation}
\label{eqk}
\exp\left( \prod\lm_{i=1}^r (1-q_i)^{-1} \right) = 1 + \sum\lm_{k=1}^{\infty}\fr{1}{k!} \prod\lm_{i=1}^r (1-q_i)^{-k},
\end{equation}
and 
$$(1-q)^{-k}=1+\sum\lm_{n=1}^{\infty}\binom{k+n-1}{k-1} q^n,$$ 
the coefficient of $\bm{q}^{\bm{\al}}$ in (\ref{eqk}) equals
\begin{equation*}
\sum\lm_{k=1}^{\infty} \fr{1}{k!} \prod\lm_{i=1}^r \binom{k+\al_i-1}{k-1} 
= \sum\lm_{k=0}^{\infty} \fr{1}{(k+1)!} \prod\lm_{i=1}^r \binom{k+\al_i}{k}. 
\end{equation*}
This completes the proof.
\end{proof}
\medskip
We are now in a position to give a proof of Proposition \ref{lb}.
\bigskip
\section{Proof of Proposition \ref{lb}}
Firstly, we have
$$ g(\bm{\al},z) = z\prod\lm_{i=1}^r \left( 1 + \fr{\al_i}{z} \right)^{-1}.$$
Taking logarithmic derivative, we find that
\begin{equation*}
\fr{g'(\bm{\al},z)}{g(\bm{\al},z)}  =  \fr{r+1}{z} - \sum\lm_{i=1}^r \fr{1}{z+\al_i} >0,
\end{equation*}
for all $z>0$. 
\vskip 0.02in
Therefore, $g(\bm{\al},z)$ is a strictly increasing function in $z$ with $g(\bm{\al},1)<1$. Hence, the equation $g(\bm{\al},z)=1$ must have a unique positive real solution $z(\bm{\al})>1$. Therefore, with $N = \floor{z(\bm{\al})} \geq 1$, one has
\begin{equation}
\label{g(N)}
g(\bm{\al},N) \leq 1 \leq g(\bm{\al},N+1).
\end{equation}
In particular, we have
\begin{equation}
\label{eq7}
\prod\lm_{i=1}^r \left(1 + \fr{\al_i}{N} \right) \geq N.
\end{equation}
\vskip 0.07in
\medskip
We now prove (a). We will use the bound given by a hypergeometric series for $p(\bm{\al})$ from Lemma \ref{hypergeometric}, namely
\begin{equation}
\label{hseries}
 p(\bm{\al}) \geq \fr{1}{e}\sum\lm_{k=0}^{\infty} \fr{1}{(k+1)!} \prod\lm_{i=1}^r \binom{k+\al_i}{k} 
= \fr{1}{e}\sum\lm_{k=0}^{\infty} T(\bm{\al},k).
\end{equation}
We do not have an asymptotic formula for this sum. Fortunately for us, the hypergeometric series converges quite rapidly and therefore only one term $T(\bm{\al},k)$ will be good enough to give a decent lower bound, provided $k$ is optimally chosen. \vskip 0.07in
Applying Lemma \ref{binofact} to $T(\bm{\al},k)$, we have for any $k \geq 1$, that
\begin{equation} 
\begin{split}
\label{T}
T(\bm{\al},k) &\geq  \fr{e^{k-1}}{2 \, k^{k+\fr{3}{2}}} 
\prod\lm_{i=1}^r \fr{1}{2\sr{2}} \fr{(k+\al_i)^{k+\al_i+\fr{1}{2}}}{\al_i^{\al_i+\fr{1}{2}} k^{k+\fr{1}{2}}}\\ 
\end{split}
\end{equation}
\medskip
We make the choice $k=N$ in (\ref{T}), to obtain
\begin{equation}
\label{eq5}
T(\bm{\al},N) \geq \fr{e^{N-1}}{2 \, N^{N + \fr{3}{2}}} 
\prod\lm_{i=1}^r \fr{1}{2\sr{2N}}  \left( 1 + \fr{\al_i}{N} \right)^N \left( 1 + \fr{N}{\al_i} \right)^{\al_i + \fr{1}{2}}.
\end{equation}
Using (\ref{eq7}) in (\ref{eq5}), we get 
\begin{equation*}
\begin{split}
p(\bm{\al}) \geq \fr{T(\bm{\al},N)}{e} \geq  \fr{e^{N-2}}{2 \, N^{\fr{3}{2}}} 
\prod\lm_{i=1}^r \fr{1}{2\sr{2N}} \left( 1 + \fr{N}{\al_i} \right)^{\al_i + \fr{1}{2}}.
\end{split}
\end{equation*}
This proves (a). \vskip 0.07in
We now prove (b). From Lemma \ref{hypergeometric}, we have
\begin{equation}
p(\bm{\al}) \geq \fr{1}{e}\sum\lm_{k=0}^{\infty} \fr{1}{(k+1)!}\prod\lm_{i=1}^r \binom{k+\al_i}{k} 
\geq \fr{1}{e}\sum\lm_{k=1}^{\infty} \fr{k^r}{k!}.
\end{equation}
Taking the term $k=\ceil{r/2}$, and using the inequality
$$ \fr{1}{k!} \geq \fr{1}{k^k}, \quad \text{for all} \ k \geq 1,$$
we obtain
$$ x \geq p(\bm{\al}) \geq \fr{1}{e} \fr{\ceil{r/2}^r}{\ceil{r/2}!} \geq \fr{1}{e}\ceil{r/2}^{\floor{r/2}}.$$
From this, it follows that $r \leq R$. \vskip 0.07in

To show $N \leq 3 \log x$, we take logarithms in (a) of Proposition \ref{lb}, to get
$$ N - 1.04 R - 0.5(R+3) \log N - \log x -2.7 \leq 0.$$
\medskip
Substituting $R$, it follows that $N \leq 3 \log x$. This completes the proof of Proposition \ref{lb}. 
\bigskip
\section{Concluding remarks}
We believe that the bound in Theorem \ref{main} is essentially the best possible due to the following reasons. Let
$$ S = \left\{\bm{\al}: \al_i \leq \sr{\log x} \ \forall \ i, \ \, \sum \al_i \leq \fr{B\log x}{\log \log x} \right\}.$$
\medskip

Then, for each $\bm{\al} \in S$, we have $p(\bm{\al}) = O(x)$. Moreover, the number of elements in this set is at least 
$\exp\left(c_1 \sr{\fr{\log x}{\log \log x}}\right)$. But we are not able to show that the values of $p(\bm{\al})$, as $\bm{\al}$ runs through $S$, are distinct. However, some calculations seem to show that the number of distinct values of $p(\bm{\al})$ above are also having a similar lower bound. We shall return to this later.
\bigskip


\end{document}